\documentclass{ws-ijnt}
\usepackage{url, hyperref, float}

\begin{document}

\markboth{M. Ddamulira \& F. Luca}
{Pillai's problem \& $ k $--generalized Fibonacci numbers}

%
\catchline{}{}{}{}{}
%

\title{On the  problem of Pillai with $k$--generalized\\ Fibonacci numbers and powers of $3$}

\author{Mahadi Ddamulira}
\address{Institute of Analysis and Number Theory\\ Graz University of Technology\\ Kopernikusgasse 24/II\\
A- 8010 Graz, Austria\,\\
\email{mddamulira@tugraz.at; mahadi@aims.edu.gh} }

\author{Florian Luca}
\address{School of Mathematics\\ University of the Witwatersrand\\ Private Bag X3, WITS 2050\\ Johannesberg, South Africa}       
\address{Research Group in Algebraic Structures and Applications\\ King Abdulaziz University\\ Jeddah, Saudi Arabia}
\address{Centro de Ciencias Matem\'aticas, UNAM \\
 Morelia, Michoac\'an, M\'exico\\ 
\email{Florian.Luca@wits.ac.za}}

\maketitle

\begin{history}
\received{22 Mayy 2019}
\revised{3 March 2020}
\accepted{3 March 2020}
Published {9 April 2020}
\end{history}

\begin{abstract}
For an integer $k\ge 2$, let $\{F^{(k)}_{n}\}_{n\ge 2-k}$ be the $ k$--generalized Fibonacci sequence which starts with $0, \ldots, 0,1$ (a total of $k$ terms) and for which each term afterwards is the sum of the $k$ preceding terms. In this paper, we find all integers $ c $ with at least two representations as a difference between a $ k $-generalized Fibonacci number and a power of $ 3 $. This paper continues the previous work of the first author for the Fibonacci numbers, and for the Tribonacci numbers.
\end{abstract}


\vspace{0.05cm}

\noindent
{\small This is an Open Access article published by World Scientific Publishing Company. It is distributed
under the terms of the Creative Commons Attribution 4.0 (CC BY) License which permits use,
distribution and reproduction in any medium, provided the original work is properly cited.}

\section{Introduction}
The problem of Pillai states that for each fixed integer $ c\ge 1 $, the Diophantine equation
\begin{eqnarray}\label{Pillai}
a^x-b^y = c, ~~~~~\min\{x,y\}\ge 2,
\end{eqnarray}
has only a finite number of positive solutions $\{a,b,x,y\}$. This problem is still open; however, the case $ c=1 $, is the conjecture of Catalan and was proved by Mih\u ailescu \cite{Mihailescu}. In 1936 (see \cite{Pillai:1936, Pillai:1937}), in the special case $ (a,b) =(3,2) $ which is a continutation of the work of Herschfeld \cite{Herschfeld:1935, Herschfeld:1936} in 1935, Pillai conjectured that the only integers $ c $ admitting at least two representations of the form $ 2^{x}-3^{y} $ are given by
\begin{eqnarray}\label{Pillaisolns}
2^{3}-3^{2}=2^{1}-3^{1}=-1, ~~~ 2^{5}-3^{3}=2^{3}-3^{1}=5, ~~~ 2^{8}-3^{5}=2^{4}-3^{1}=13.
\end{eqnarray}
This  was confirmed by Stroeker and Tijdeman \cite{StroekerTijdeman} in 1982. The general problem of Pillai is difficult to solve and this has motivated the consideration of special cases of this problem. In the past years, several special cases of the problem of Pillai have been studied. See, for example, \cite{Luca16, Chim1, Chim2, Ddamulira11, Ddamulira12, Ddamulira1, Hernane, Hernandez}.

Let $k \geqslant 2$ be an integer. We consider a generalization of Fibonacci sequence called the $k$--generalized Fibonacci sequence $\lbrace F_n^{(k)} \rbrace_{n\geqslant 2-k}$ defined as
\begin{equation}
F_n^{(k)} = F_{n-1}^{(k)} + F_{n-2}^{(k)} + \cdots + F_{n-k}^{(k)},\label{fibb1}
\end{equation}
with the initial conditions
\[F_{-(k-2)}^{(k)} = F_{-(k-3)}^{(k)} = \cdots = F_{0}^{(k)} = 0 \quad {\text {\rm  and  }}\quad F_{1}^{(k)} = 1.\]
We call $F_{n}^{(k)}$ the $n$th $k$--generalized Fibonacci number. Note that when $k=2$, it  coincides with the Fibonacci numbers and when $k=3$ it is the Tribonacci number. The first $k+1$ nonzero terms in $F_{n}^{(k)}$ are powers of $2$, namely
\begin{equation*}
F_{1}^{(k)}=1,\quad F_{2}^{(k)}=1,\quad F_{3}^{(k)}=2, \quad F_{4}^{(k)}=4,\ldots, F_{k+1}^{(k)}=2^{k-1}.
\end{equation*}
Furthermore, the next term is $F_{k+2}^{(k)}=2^{k}-1$.
Thus, we have that
\begin{equation}\label{Fibbo111}
F_{n}^{(k)} = 2^{n-2} \quad {\text{\rm holds for all}}\quad 2\leq n\leq k+1.
\end{equation}
We also observe that the recursion \eqref{fibb1} implies the three--term recursion
$$
F_{n}^{(k)} =2F_{n-1}^{(k)} - F_{n-k-1}^{(k)} \quad {\text{\rm for all}}\quad n\geq 3,\label{fibb2}
$$
which can be used to prove by induction on $m$ that $F_n^{(k)}<2^{n-2}$ for all $n\geq k+2$ (see also \cite{BBL17}, Lemma $2$).

The generalised Fibonacci analogue of the problem of Pillai under the same conditions as in \eqref{Pillai}, concerns studying for fixed $(k,\ell)$ all values of the integer $c$ such that the equation
\begin{eqnarray}\label{Pillai2}
F_n^{(k)} - F_m^{(\ell)} = c
\end{eqnarray}
has at least two solutions $(n,m)$. We are not aware of a general treatment of equation \eqref{Pillai2} (namely, considering $k$ and $\ell$ parameters), although the particular case when $\{k,\ell\}=\{2,3\}$ was treated in \cite{Chim1}.

Ddamulira, G\'omez and Luca \cite{Ddamulira2}, studied the Diophantine equation
\begin{eqnarray} \label{Pillai3}
F_{n}^{(k)}-2^{m}=c,
\end{eqnarray}
where $k$ is also a parameter, which is a variation of equation \eqref{Pillai2}. They determined all integers $c$ such that equation \eqref{Pillai3} has at least two solutions $(n,m)$. 
These $c$ together with their multiple representations as in \eqref{Pillai3} turned out to be grouped into four parametric families. 

In this paper, we study a related problem and we find all integers $c$ admitting at least two representations of the form $ F_n^{(k)} - 3^m $ for some positive integers $k$, $n$ and $m$. This can be interpreted as solving the equation
\begin{align}
\label{Problem}
 F_n^{(k)} - 3^m = F_{n_1}^{(k)} - 3^{m_1}~~~~(=c)
\end{align}
with $(n, m) \neq (n_1, m_1)$. The cases $k=2$ and $ k=3 $ have been solved completely by the first author in \cite{Ddamulira11} and \cite{Ddamulira12}, respectively. So, we focus on the case  $k \geqslant 4$.\\

\begin{theorem}\label{Main}
For fixed integer $ k\ge 4 $, the Diophantine equation \eqref{Problem} with $ n>n_1\geq 2 $ and $ m>m_1\ge 1 $ has:
\begin{itemize}
\item[(i)] solutions with $c\in \{-1, 5, 13\}$ and $ 2\le n \le k+1 $, which arise from the classical Pillai problem for $(a,b)=(2,3)$, namely: 
\begin{eqnarray*}
F_{5}^{(k)}-3^{2}&=&F_{3}^{(k)}-3^{1} ~~=~ -1, ~~~k\ge 4, \\
F_{7}^{(k)}-3^{3}&=&F_{5}^{(k)}-3^{1} ~~=~~ 5, ~~~k\ge 6,\\
F_{10}^{(k)}-3^{5}&=&F_{6}^{(k)}-3^{1} ~~=~~ 13, ~~~k\ge 9;
\end{eqnarray*}
\item[(ii)] solutions with $c\in \{-25, -7,5\} $ and $ n\geq k+2 $ and $ k\in \{4,5,6\} $. Futhermore, all the representations of $ c $ in this case are given by
\begin{eqnarray*}
F_{8}^{(4)}-3^{4}&=&F_{3}^{(4)}-3^{3} ~~=~ -25,\\
F_{10}^{(5)}-3^{5}&=&F_{3}^{(5)}-3^{2} ~~=~ -7,\\
F_{10}^{(6)}-3^5&=&F_6^{(6)}-3^1~~=~~5.
\end{eqnarray*}
for $ k=4,5 $ and $6$, respectively.
\end{itemize}
\end{theorem}

\section{Preliminary Results}

In this section, we recall some general results from algebra number theory and diophantine approximations 
and properties of the $k$-generalized Fibonacci sequence.

\subsection{Notations and terminology from algebraic number theory} 

We begin by recalling some basic notions from algebraic number theory.

Let $\eta$ be an algebraic number of degree $d$ with minimal primitive polynomial over the integers
$$
a_0x^{d}+ a_1x^{d-1}+\cdots+a_d = a_0\prod_{i=1}^{d}(x-\eta^{(i)}),
$$
where the leading coefficient $a_0$ is positive and the $\eta^{(i)}$'s are the conjugates of $\eta$. Then the \textit{logarithmic height} of $\eta$ is given by
$$ 
h(\eta):=\dfrac{1}{d}\left( \log a_0 + \sum_{i=1}^{d}\log\left(\max\{|\eta^{(i)}|, 1\}\right)\right).
$$
In particular, if $\eta=p/q$ is a rational number with $\gcd (p,q)=1$ and $q>0$, then $h(\eta)=\log\max\{|p|, q\}$. The following are some of the properties of the logarithmic height function $h(\cdot)$, which will be used in the next sections of this paper without reference:
\begin{eqnarray}
h(\eta\pm \gamma) &\leq& h(\eta) +h(\gamma) +\log 2,\nonumber\\
h(\eta\gamma^{\pm 1})&\leq & h(\eta) + h(\gamma),\\
h(\eta^{s}) &=& |s|h(\eta) \qquad (s\in\mathbb{Z}). \nonumber
\end{eqnarray}

\subsection{$k$-generalized Fibonacci numbers}

It is known that the characteristic polynomial of the $k$--generalized Fibonacci numbers $F^{(k)}:=(F_m^{(k)})_{m\geq 2-k}$, namely
$$
\Psi_k(x) := x^k - x^{k-1} - \cdots - x - 1,
$$
is irreducible over $\mathbb{Q}[x]$ and has just one root outside the unit circle. Let $\alpha := \alpha(k)$ denote that single root, which is located between $2\left(1-2^{-k} \right)$ and $2$ (see \cite{Dresden}). 
This is called the dominant root of $F^{(k)}$. To simplify notation, in our application we shall omit the dependence on $k$ of $\alpha$. We shall use $\alpha^{(1)}, \dotso, \alpha^{(k)}$ for all roots of $\Psi_k(x)$ with the convention that $\alpha^{(1)} := \alpha$.

We now consider for an integer $ k\geq 2 $, the function
\begin{eqnarray}\label{fun12}
f_{k}(z) = \dfrac{z-1}{2+(k+1)(z-2)} \qquad {\text{for}}\qquad z \in \mathbb{C}.
\end{eqnarray}
With this notation, Dresden and Du presented in  \cite{Dresden} the following ``Binet--like" formula for the terms of $F^{(k)}$:
\begin{eqnarray} \label{Binet}
F_m^{(k)} = \sum_{i=1}^{k} f_{k}(\alpha^{(i)}) {\alpha^{(i)}}^{m-1}.
\end{eqnarray}
It was proved in \cite{Dresden} that the contribution of the roots which are inside the unit circle to the formula (\ref{Binet}) is very small, namely that the approximation
\begin{equation} \label{approxgap}
\left| F_m^{(k)} - f_{k}(\alpha)\alpha^{m-1} \right| < \dfrac{1 }{2} \quad \mbox{holds~ for~ all~ } m \geqslant 2 - k.
\end{equation}
It was proved by Bravo and Luca in \cite{BBL17} that
\begin{eqnarray}\label{Fib12}
\alpha^{m-2} \leq F_{m}^{(k)} \leq \alpha^{m-1}\qquad {\text{\rm holds for all}}\qquad m\geq 1\quad {\text{\rm and}}\quad k\geq 2.
\end{eqnarray}

Before we conclude this section, we present some useful lemma that will be used in the next sections on this paper. The following lemma was proved by Bravo and Luca in \cite{BBL17}.
\begin{lemma}[Bravo, Luca]\label{fala5}
Let $k\geq 2$, $\alpha$ be the dominant root of $\{F^{(k)}_m\}_{m\ge 2-k}$, and consider the function $f_{k}(z)$ defined in \eqref{fun12}. 
\begin{itemize}
\item[(i)]\label{kat1} The inequalities
$$
\dfrac{1}{2}< f_{k}(\alpha)< \dfrac{3}{4}\qquad \text{and}\qquad |f_{k}(\alpha^{(i)})|<1, \qquad  2\leq i\leq k
$$
hold. In particular, the number $f_{k}(\alpha)$ is not an algebraic integer.
\item[(ii)]\label{kat2}The logarithmic height of $f_k(\alpha)$ satisfies $h(f_{k}(\alpha))< 3\log k$.
\end{itemize}
\end{lemma}

\subsection{Linear forms in logarithms and continued fractions}

In order to prove our main result Theorem \ref{Main}, we need to use several times a Baker--type lower bound for a nonzero linear form in logarithms of algebraic numbers. There are many such in the literature  like that of Baker and W{\"u}stholz from \cite{bawu07}. We use the following result by Matveev \cite{MatveevII}, which is one of our main tools in this paper.

\begin{theorem}[Matveev]\label{Matveev11} Let $\gamma_1,\ldots,\gamma_t$ be positive real algebraic numbers in a real algebraic number field 
$\mathbb{K}$ of degree $D$, $b_1,\ldots,b_t$ be nonzero integers, and assume that
\begin{equation}
\label{eq:Lambda}
\Lambda:=\gamma_1^{b_1}\cdots\gamma_t^{b_t} - 1
\end{equation}
is nonzero. Then
$$
\log |\Lambda| > -1.4\times 30^{t+3}\times t^{4.5}\times D^{2}(1+\log D)(1+\log B)A_1\cdots A_t,
$$
where
$$
B\geq\max\{|b_1|, \ldots, |b_t|\},
$$
and
$$A
_i \geq \max\{Dh(\gamma_i), |\log\gamma_i|, 0.16\},\qquad {\text{for all}}\qquad i=1,\ldots,t.
$$
\end{theorem} 

During the course of our calculations, we get some upper bounds on our variables which are too large, thus we need to reduce them. To do so, we use some results from the theory of continued fractions. Specifically, for a nonhomogeneous linear form in two integer variables, we use a slight variation of a result due to Dujella and Peth{\H o} (see \cite{dujella98}, Lemma 5a), which  itself is a generalization of a result of Baker and Davenport \cite{BD69}.

For a real number $X$, we write  $||X||:= \min\{|X-n|: n\in\mathbb{Z}\}$ for the distance from $X$ to the nearest integer.
\begin{lemma}[Dujella, Peth\H o]\label{Dujjella}
Let $M$ be a positive integer, $p/q$ be a convergent of the continued fraction of the irrational number $\tau$ such that $q>6M$, and  $A,B,\mu$ be some real numbers with $A>0$ and $B>1$. Let further 
$\varepsilon: = ||\mu q||-M||\tau q||$. If $ \varepsilon > 0 $, then there is no solution to the inequality
$$
0<|u\tau-v+\mu|<AB^{-w},
$$
in positive integers $u,v$ and $w$ with
$$ 
u\le M \quad {\text{and}}\quad w\ge \dfrac{\log(Aq/\varepsilon)}{\log B}.
$$
\end{lemma}

The above lemma cannot be applied when $\mu=0$ (since then $\varepsilon<0$). In this case, we use the following criterion of Legendre.

\begin{lemma}[Legendre]
\label{lem:legendre}
Let $\tau$ be real number and $x,y$ integers such that
\begin{equation}
\label{eq:continuedfraction}
\left|\tau-\frac{x}{y}\right|<\frac{1}{2y^2}.
\end{equation}
Then $x/y=p_k/q_k$ is a convergent of $\tau$. Furthermore, 
\begin{equation}
\label{eq:continuedfraction1}
\left|\tau-\frac{x}{y}\right|\ge \frac{1}{(a_{k+1}+2)y^2}.
\end{equation}
\end{lemma}

Finally, the following lemma is also useful. It is Lemma 7 in \cite{guzmanluca}. 

\begin{lemma}[G\'uzman, Luca]
\label{gl}
If $m\geqslant 1$, $T>(4m^2)^m$  and $T>x/(\log x)^m$, then
$$
x<2^mT(\log T)^m.
$$
\end{lemma}

\section{The connection with the classical Pillai problem}

Assume that $ (n,m)\neq (n_{1}, m_{1}) $ are such that
$$F_{n}^{(k)} - 3^{m} = F_{n_{1}}^{(k)}-3^{m_{1}}.$$
If $ m=m_{1} $, then $ F_{n}^{(k)} = F_{n_{1}}^{(k)} $ and since $ \min\{n,n_{1}\}\geq 2$, we get that $ n=n_{1} $. Thus, $ (n,m) = (n_{1}, m_{1}) $, contradicting our assumption. Hence, $ m\neq m_{1} $, and we may assume without loss of generality that $ m>m_{1}\geq 1 $. Since
\begin{eqnarray}
\label{fala1}
F_{n}^{(k)} - F_{n_{1}}^{(k)} &=& 3^{m}-3^{m_{1}},
\end{eqnarray}
and the right--hand side of \eqref{fala1} is positive, we get that the left--hand side of
\eqref{fala1} is also positive and so $ n>n_{1} $. Furthermore, since $ F_{1}^{(k)}=F_{2}^{(k)}=1 $,   we may assume that $ n> n_{1}\geq 2 $. 

\medskip
We analyse the possible situations.
\medskip

\noindent {\bf Case 1.}
Assume that $ 2\leq n_{1}<n\leq k+1 $. Then, by \eqref{Fibbo111}, we have
\begin{eqnarray*}
F_{n_{1}}^{(k)}=2^{n_{1}-2}\qquad {\rm and} \qquad F_{n}^{(k)}=2^{n-2}
\end{eqnarray*}
so, by substituting them in \eqref{fala1}, we get
\begin{eqnarray*}
2^{n-2}-3^{m} = 2^{n_1-2}-3^{m_{1}}.
\end{eqnarray*}
By comparing with the classical solutions in \eqref{Pillaisolns}, and by using the fact that $F_n^{(k)}$ is a power of $2$ if and only if $n\le k+1$ (see \cite{BravoLucapowersof2}), we get the solutions
\begin{eqnarray}
F_{5}^{(k)}-3^{2}&=&F_{3}^{(k)}-3^{1} ~~=~ -1, ~~~k\ge 4,\nonumber \\
F_{7}^{(k)}-3^{3}&=&F_{5}^{(k)}-3^{2} ~~=~~ 5, ~~~k\ge 6,\\
F_{10}^{(k)}-3^{5}&=&F_{6}^{(k)}-3^{1} ~~=~~ 13, ~~~k\ge 9.\nonumber
\end{eqnarray}

\medskip

\noindent {\bf Case 2.}
Assume $ n \ge k+2 $. The following lemma is useful.
\begin{lemma}\label{classical1}
For $n\ge k+2$, the conditions 
$$
F_n^{(k)}-3^m=F_{n_1}^{(k)}-3^{m_1}\quad {\text{and}}\quad 2^{n-2}-3^m=2^{n_1-2}-3^{m_1}
$$
cannot simultaneously hold.
\end{lemma}

\begin{proof} If they do, then
$$
2^{n-2}-F_{n}^{(k)}=2^{n_1-2}-F_{n_1}^{(k)}.
$$
The sequence $\{2^{n-2}-F_{n}^{(k)}\}_{n\ge 2}$ is $0$ at $n=2,3,\ldots,k+1$ and is $1$ at $n=k+2$. We show that from here on it is increasing. That is 
$$
2^{n-1}-F_{n+1}^{(k)}>2^{n-2}-F_n^{(k)}\quad {\text{\rm holds~for}}\quad n\ge k+2.
$$
This is equivalent to
$$
2^{n-2}>F_{n+1}^{(k)}-F_n^{(k)}=F_{n-1}^{(k)}+\cdots+F_{n+1-k}^{(k)},
$$
and this last inequality holds true because in the right--hand side we have $F_i\le 2^{i-2}$ for $i=n+1-k,n+2-k,\ldots,n-1$ and then 
$$
\sum_{i=n-k+1}^{n-1} F_i\le \sum_{i=n-k+1}^{n-1} 2^{i-2}<1+2+\cdots+2^{n-3}<2^{n-2}.
$$
\end{proof}

\section{Bounding $n$ in terms of $m$ and $k$}

By the results of the previous section, we assume that $n\ge k+2$. Thus, $ 2^{n-2}-3^{m}\neq 2^{n_1-2}-3^{m_1} $. Since $ n>n_{1}\geq 2 $, we have that $ F_{n_{1}}^{(k)}  \leq F_{n-1}^{(k)}$ and therefore
$$F_{n}^{(k)} = F_{n-1}^{(k)} + \cdots + F_{n-k}^{(k)} \geq F_{n-1}^{(k)} + \cdots + F_{n-k-1}^{(k)} \geq F_{n_{1}}^{(k)} + \cdots + F_{n-k-1}^{(k)}.$$
So, from the above, \eqref{Fib12} and \eqref{fala1}, we have
\begin{eqnarray}\label{fala4}
\alpha^{n-4}&\leq& F_{n-2}^{(k)}\leq F_{n}^{(k)}-F_{n_{1}}^{(k)} = 3^{m}-3^{m_{1}} < 3^{m}, \text{  and  }\\
\alpha^{n-1}&\geq& F_{n}^{(k)} > F_{n}^{(k)}-F_{n_{1}}^{(k)} = 3^{m}-3^{m_{1}}\geq 3^{m-1},\nonumber
\end{eqnarray}
leading to
\begin{eqnarray}\label{fala2}
1+\left(\dfrac{\log 3}{\log\alpha}\right) (m - 1) < n< \left(\dfrac{\log 3}{\log \alpha}\right)m + 4.
\end{eqnarray}

We note that the above inequality \eqref{fala2} in particular implies that $ m < n< 1.6m+4 $. 
We assume for technical reasons that $ n>600 $. By \eqref{approxgap} and \eqref{fala1}, we get
\begin{eqnarray*}
\left|f_{k}(\alpha)\alpha^{n-1} - 3^{m}\right|&=&\left|(f_{k}(\alpha)\alpha^{n-1}-F_{n}^{(k)})+(F_{n_{1}}^{(k)}-3^{m_{1}})\right|\\
&=&\left|(f_{k}(\alpha)\alpha^{n-1}-F_{n}^{(k)})+(F_{n_{1}}^{(k)}-f_{k}(\alpha)\alpha^{n_{1}-1})+(f_{k}(\alpha)\alpha^{n_{1}-1}-3^{m_{1}})\right|\\
&<&\dfrac{1}{2}+\dfrac{1}{2}+\alpha^{n_{1}-1}+3^{m_{1}}\\
&<& \alpha^{n_{1}}+3^{m_{1}}\\
&<& 2\max\{\alpha^{n_{1}}, 3^{m_{1}}\}.
\end{eqnarray*}
In the above, we have also used the fact that $ |f_{k}(\alpha)| < 1 $ (see Lemma \ref{fala5}). Dividing through by $ 3^{m} $, we get
\begin{eqnarray}\label{fala3}
&&\left|f_{k}(\alpha)\alpha^{n-1}3^{-m} -1\right|< 2\max\left\{\dfrac{\alpha^{n_{1}}}{3^{m}}, 3^{m_{1}-m}\right\} < \max\{\alpha^{n_{1}-n+6}, 3^{m_{1}-m+1}\},
\end{eqnarray}
where for the right--most inequality in \eqref{fala3} we used \eqref{fala4} and the fact that $ \alpha^{2}> 2 $.

For the left-hand side of \eqref{fala3} above, we apply Theorem \ref{Matveev11} with the data
$$
t:=3, \quad \gamma_{1}:=f_{k}(\alpha), \quad \gamma_{2}: = \alpha,\quad \gamma_{3}:=3, \quad b_{1}:=1, \quad b_{2}:=n-1, \quad b_{3}:=-m .
$$
We begin by noticing that the three numbers $ \gamma_{1}, \gamma_{2}, \gamma_{3} $ are positive real numbers and belong to the field $ \mathbb{K}: = \mathbb{Q}(\alpha)$, so we can take $ D:= [\mathbb{K}:\mathbb{Q}]= k$. Put
$$\Lambda :=f_{k}(\alpha)\alpha^{n-1}3^{-m} -1. $$
To see why $ \Lambda \neq 0 $, note that otherwise, we would then have that $ f_{k}(\alpha) = 3^{m}\alpha^{-(n-1)} $ and so $ f_{k}(\alpha) $ would be an algebraic integer, which contradicts Lemma \ref{fala5} (i).

Since $ h(\gamma_{2})= (\log \alpha)/k <(\log 2)/k $ and $ h(\gamma_{3})= \log 3$, it follows that we can take $ A_{2}:= \log 2$ and $ A_{3}:= k\log 3 $. Further, in view of Lemma \ref{fala5} (ii), we have that $ h(\gamma_{1})<3\log k$, so we can take $ A_{1}:=3k\log k $. Finally, since $ \max\{1, n-1, m\} = n-1$, we take $ B:=n $.

Then, the left--hand side of \eqref{fala3} is bounded below, by Theorem \ref{Matveev11}, as
$$\log |\Lambda| >-1.4\times 30^6 \times 3^{4.5} \times k^4 (1+\log k)(1+\log n)(3\log k)(\log 2)(\log 3).$$
Comparing with \eqref{fala3}, we get
$$\min\{(n-n_{1}-6)\log\alpha , (m-m_{1}-1)\log 3\} < 6.54\times 10^{11} k^4 \log^{2}k(1+\log n),$$
which gives
$$\min\{(n-n_{1})\log\alpha , (m-m_{1})\log 3\} < 6.60\times 10^{11} k^4 \log^{2}k(1+\log n).$$
Now the argument is split into two cases.\\

\textbf{Case 1.} $\min \lbrace (n-n_1) \log \alpha , (m-m_1) \log 2 \rbrace  = (n - n_{1}) \log \alpha$.

\medskip

In this case, we rewrite \eqref{fala1} as
\begin{eqnarray*}
\left| f_{k}(\alpha)\alpha^{n-1} - f_{k}(\alpha)\alpha^{n_{1}-1} - 3^{m}\right| &=& \left|(f_{k}(\alpha)\alpha^{n-1}- F_{n}^{(k)}) + (F_{n_{1}}^{(k)} - f_{k}(\alpha)\alpha^{n_{1}-1}) - 3^{m_{1}}\right|\\
&<&\dfrac{1}{2}+\dfrac{1}{2}+3^{m_{1}} \leq 3^{m_{1}+1}.
\end{eqnarray*}
Dividing through by $ 3^{m} $ gives
\begin{eqnarray}\label{fala6}
\left| f_{k}(\alpha)(\alpha^{n-n_{1}}-1)\alpha^{n_{1}-1}3^{-m} - 1\right|&<&3^{m_{1}-m+1}.
\end{eqnarray}
Now we put
$$
\Lambda_{1} := f_{k}(\alpha)(\alpha^{n-n_{1}}-1)\alpha^{n_{1}-1}3^{-m} - 1.
$$
We apply again Theorem \ref{Matveev11} with the following data
$$
t:=3,\quad \gamma_{1} :=f_{k}(\alpha)(\alpha^{n-n_{1}}-1), \quad \gamma_{2}:=\alpha, \quad \gamma_{3}:=3,  \quad b_{1}:=1, \quad b_{2}:=n_{1}-1, \quad b_{3}:=-m.
$$
As before, we begin by noticing that the three numbers $ \gamma_{1}, \gamma_{2}, \gamma_{3} $ belong to the field $ \mathbb{K} := \mathbb{Q}(\alpha) $, so we can take $ D:= [\mathbb{K}: \mathbb{Q}] = k$. To see why $ \Lambda_{1} \neq 0$, note that otherwise, we would get the relation $ f_{k}(\alpha)(\alpha^{n-n_{1}}-1) = 3^{m}\alpha^{1-n_{1}} $. Conjugating this last equation with any automorphism $ \sigma$ of the Galois group of $ \Psi_{k}(x) $ over $ \mathbb{Q} $ such that $ \sigma(\alpha) = \alpha^{(i)} $ for some $ i\geq 2 $, and then taking absolute values, we arrive at the equality $ |f_{k}(\alpha^{(i)})((\alpha^{(i)})^{n-n_1}-1)| = |3^{m}(\alpha^{(i)})^{1-n_1}| $. But this cannot hold because, $ |f_{k}(\alpha^{(i)})||(\alpha^{(i)})^{n-n_1}-1|<2 $ since $ |f_{k}(\alpha^{(i)})|<1 $ by Lemma \ref{fala5} (i), and $ |(\alpha^{(i)})^{n-n_1}|<1 $, since $ n>n_1$, while $ |3^{m}(\alpha^{(i)})^{1-n_1}|\geq 3$.

Since
$$
h(\gamma_{1})\leq h(f_{k}(\alpha)) +h(\alpha^{n-n_{1}}-1) 
< 3\log k +(n-n_{1})\dfrac{\log\alpha}{k}+\log 2,
$$
it follows that
$$
kh(\gamma_{1}) < 6k\log k + (n - n_1)\log\alpha < 6k\log k + 6.60 \times 10^{11} k^4 \log^{2}k(1+\log n).
$$
So, we can take $ A_{1}:= 6.80\times 10^{11} k^4 \log^{2}k(1+\log n) $. Further, as before, we take $ A_{2} :=\log 2 $ and $ A_{3}: = k\log3 $. Finally, by recalling that $ m<n $, we can take $ B:=n $.

We then get that
$$\log|\Lambda_{1}|>-1.4\times 30^6 \times 3^{4.5}\times k^{3}(1+\log k)(1+\log n)(6.80\times 10^{11} k^4 \log^{2}k(1+\log n))(\log 2)(\log 3),$$
which yields
$$ \log |\Lambda_{1}|>-7.41 \times 10^{22} k^7\log^3 k(1+\log n)^2.$$
Comparing this with \eqref{fala6}, we get that
$$(m-m_{1})\log 3 < 7.50\times 10^{22} k^7\log^3 k(1+\log n)^{2}.$$

\medskip

\textbf{Case 2.} $\min \lbrace (n-n_1) \log \alpha , (m-m_1) \log 3 \rbrace  = (m - m_{1} ) \log 3$.

\medskip

In this case, we write \eqref{fala1} as
\begin{eqnarray*}
\left|f_{k}(\alpha)\alpha^{n-1} - 3^{m} +3^{m_{1}}\right| &=& \left|(f_{k}(\alpha)\alpha^{n-1} -F_{n}^{(k)}) + (F_{n_{1}}^{(k)} - f_{k}(\alpha)\alpha^{n_{1}-1})+ f_{k}(\alpha)\alpha^{n_{1}-1} \right| \\
&<&\dfrac{1}{2}+\dfrac{1}{2}+\alpha^{n_{1}-1} ~~<~~\alpha^{n_{1}},
\end{eqnarray*}
so that
\begin{eqnarray}\label{fala7}
&&\left|f_{k}(\alpha)(3^{m-m_{1}}-1)^{-1}\alpha^{n-1}3^{-m_{1}} - 1\right|<\dfrac{\alpha^{n_{1}}}{3^{m}-3^{m_{1}}}\leq \dfrac{2\alpha^{n_{1}}}{3^{m}}<\alpha^{n_{1}-n+6}.
\end{eqnarray}
The above inequality \eqref{fala7} suggests once again studying a lower bound for the absolute value of
$$
\Lambda_{2} := f_{k}(\alpha)(3^{m-m_{1}}-1)^{-1}\alpha^{n-1}3^{-m_{1}} - 1.
$$
We again apply Matveev's theorem with the following data
$$
t: =3,\quad \gamma_{1}: =f_{k}(\alpha)(3^{m-m_{1}}-1)^{-1},\quad \gamma_{2}: = \alpha, \quad \gamma_{3}: = 3, \quad b_{1}:=1,\quad b_{2}:=n-1, \quad b_{3}:=-m_{1}.
$$
We can again take $ B:=n $ and  $ \mathbb{K} := \mathbb{Q}(\alpha) $, so that $ D:=k $. We also note that, if $ \Lambda_{2} =0 $, then $ f_{k}(\alpha) = \alpha^{-(n-n_{1})} 3^{m_{1}} (3^{m-m_{1}}-1) $ implying that $ f_{k}(\alpha) $ is an algebraic integer, which is not the case. Thus, $ \Lambda_{2} \neq 0 $.

Now, we note that
$$
h(\gamma_{1})\leq  h(f_{k}(\alpha))+h(3^{m-m_{1}}-1)
<3\log k +(m-m_{1}+k)\dfrac{\log 3}{k}.
$$
Thus, $ kh(\gamma_{1})< 4k\log k + (m-m_{1})\log 3  < 6.80 \times 10^{11} k^4\log^2k(1+\log n)$, and so we can take $ A_{1} := 6.80 \times 10^{11} k^4\log^2k(1+\log n) $. As before, we take $ A_{2}: = \log 2 $ and $ A_{3} := k\log  3 $.
It then follows from Matveev's theorem, after some calculations, that
$$
\log |\Lambda_{2}| > -7.41\times 10^{22}k^7\log^3 k(1+\log n)^2.
$$
From this and \eqref{fala7}, we obtain that
$$(n-n_{1})\log\alpha < 7.50\times 10^{22}k^7\log^3 k(1+\log n)^2.$$
Thus, in both Case $ 1 $ and Case $ 2 $, we have
\begin{eqnarray}\label{fala8}
\min\{(n-n_{1})\log\alpha , (m-m_{1})\log 2\} & < & 6.6\times 10^{11} k^4 \log^{2}k(1+\log n),\\
\max\{(n-n_{1})\log\alpha , (m-m_{1})\log 2\} & < & 7.5\times 10^{22}k^7\log^3 k(1+\log n)^2.\nonumber
\end{eqnarray}
We now finally rewrite equation \eqref{fala1} as
$$
\left|f_{k}(\alpha)\alpha^{n-1} -f_{k}(\alpha)\alpha^{n_{1}-1}-3^{m}+3^{m_{1}}\right| = \left|(f_{k}(\alpha)\alpha^{n-1} - F_{n}^{(k)})+(F_{n_{1}}^{(k)} - f_{k}(\alpha)\alpha^{n_{1}-1})\right| < 1.
$$
We divide through both sides by $ 3^{m}-3^{m_{1}} $ getting
\begin{eqnarray}\label{fala9}
&& \left|\dfrac{f_{k}(\alpha)(\alpha^{n-n_{1}}-1)}{3^{m-m_{1}}-1}\alpha^{n_{1}-1}3^{-m_{1}} - 1\right|<\dfrac{1}{3^{m}-3^{m_{1}}} \leq \dfrac{2}{3^{m}} <3^{5 - 0.8n},
\end{eqnarray}
since $n < 1.6m+4$. To find a lower--bound on the left--hand side of \eqref{fala9} above, we again apply Theorem \ref{Matveev11} with the data
$$
t:=3,\quad \gamma_{1}: =\dfrac{f_{k}(\alpha)(\alpha^{n-n_{1}}-1)}{3^{m-m_{1}}-1},\quad \gamma_{2} := \alpha, \quad \gamma_{3} := 3, \quad b_{1}:=1,\quad b_{2}:=n_{1}-1, \quad b_{3}:=-m_{1}.
$$
We also take $ B:=n $ and we take  $ \mathbb{K} := \mathbb{Q}(\alpha) $ with $ D := k $. From the properties of the logarithmic height function, we have that
\begin{eqnarray*}
kh(\gamma_{1})&\leq& k\left(h(f_{k}(\alpha))+h(\alpha^{n-n_1}-1)+h(3^{m-m_{1}}-1)\right)\\
&<&3k\log k +(n-n_{1})\log\alpha +k(m-m_{1})\log3 + 2k\log2\\
&<&8.3\times 10^{22}k^8\log^3 k(1+\log n)^2,
\end{eqnarray*}
where in the above chain of inequalities we used the bounds \eqref{fala8}. So we can take $ A_{1} :=8.3\times 10^{22}k^8\log^3 k(1+\log n)^2  $, and certainly as before we take $ A_{2} := \log 2 $ and $ A_{3}: = k\log  3 $. We need to show that if we put
$$
\Lambda_{3}:=\dfrac{f_{k}(\alpha)(\alpha^{n-n_{1}}-1)}{3^{m-m_{1}}-1}\alpha^{n_{1}-1}3^{-m_{1}} - 1,
$$
then $ \Lambda_{3} \neq 0 $. To see why $ \Lambda_{3} \neq 0$, note that otherwise, we would get the relation 
$$ 
f_{k}(\alpha)(\alpha^{n-n_{1}}-1) = 3^{m_{1}}\alpha^{1-n_{1}}(3^{m-m_{1}}-1).
$$ 
Again, as for the case of $ \Lambda_{1} $, conjugating the above relation with an automorphism $ \sigma $ of the Galois group of $ \Psi_{k}(x) $ over $ \mathbb{Q} $ such that $ \sigma(\alpha) = \alpha^{(i)} $ for some $ i\geq 2 $, and then taking absolute values, we get that $ |f_{k}(\alpha^{(i)})((\alpha^{(i)})^{n-n_1}-1)| = |3^{m_1}(\alpha^{(i)})^{1-n_1}(3^{m-m_{1}}-1)| $. This cannot hold true because in the left--hand side we have $ |f_{k}(\alpha^{(i)})||(\alpha^{(i)})^{n-n_1}-1|<2 $, while in the right--hand side we have $ |3^{m_{1}}||(\alpha^{(i)})^{1-n_1}||3^{m-m_1}-1|\geq 4 $. Thus,
$ \Lambda_{3} \neq 0 $. Then Theorem \ref{Matveev11} gives
$$\log |\Lambda_{3}|>-1.4\times 30^{6}\times 3^{4.5}k^{11}(1+\log k)(1+\log n)\left(8.3\times 10^{22}\log^3 k(1+\log n)^2\right)(\log 2)(\log 3),$$
which together with \eqref{fala9} gives
$$(0.8n - 5)\log 3 < 9.05\times 10^{33} k^{11}\log^{4}k(1+\log n)^{3}.$$
The above inequality leads to
\begin{eqnarray*}
n < 6.2\times 10^{34} k^{11}\log^{4}k\log^{3}n,
\end{eqnarray*}
which can be equivalently written as
\begin{eqnarray}\label{fala10}
\dfrac{n}{(\log n)^{3}} & < & 6.2\times 10^{34} k^{11}\log^{4}k.
\end{eqnarray}
We apply Lemma \ref{gl} with the data $ m=3, ~~ x=n, ~~ T=6.2\times 10^{34} k^{11}\log^{4}k  $. Inequality \eqref{fala10} yields
\begin{eqnarray}
n & < & 8\times(6.2\times 10^{34} k^{11}\log^{4}k) \log (6.2\times 10^{34} k^{11}\log^{4}k)^{3}\nonumber\\
&<&4\times 10^{42}k^{11}(\log k)^{7}.
\end{eqnarray}
We then record what we have proved so far as a lemma.
\begin{lemma}\label{lemmaBD}
If $ (n,m,n_{1},m_{1}, k) $ is a solution in positive integers to equation \eqref{Problem}  with $ (n,m)\neq (n_{1},m_{1}) $, $ n> \min\{k+2, n_1+1\}$, $n_{1}\geq 2 $, $ m>m_{1}\geq 1 $ and $ k\geq 4 $, we then have that $ n < 4\times 10^{42}k^{11}(\log k)^{7}$.
\end{lemma}

\section{Reduction of the bounds on $ n $}

\subsection{The cutoff $k$}

We have from the above lemma that Baker's method gives
$$
n < 4\times 10^{42}k^{11}(\log k)^7.
$$
By imposing that the above amount is at most $ 2^{k/2} $, we get
\begin{eqnarray*}
4\times 10^{42}k^{11}(\log k)^7 &<& 2^{k/2}. 
\end{eqnarray*}
The inequality above holds for $k>600$. 


We now reduce the bounds and to do so we make use of Lemma \ref{Dujjella} several times.

\subsection{The Case of small $ k $}
We now treat the cases when $ k\in [4, 600] $. First, we consider equation \eqref{fala1} which is equivalent to \eqref{Problem}. For $ k\in[4, 600] $ and $ n\in[3, 600] $, consider the sets
\begin{eqnarray*}
F_{n,k}:=\left\{F_{n}^{(k)}-F_{n_1}^{(k)} (\text{mod} ~10^{20}):~n\in [3, 600], ~ n_1 \in [2, n-1]\right\}
\end{eqnarray*}
and
\begin{eqnarray*}
D_{n,k}:=\left\{3^{m}-3^{m_1}(\text{mod}~ 10^{20}): ~m\in [2, 600], ~m_1 \in [1, m-1]\right\}.
\end{eqnarray*}
With the help of \textit{Mathematica}, we intersected these two sets and found the only solutions listed in Theorem \ref{Main}.

 Next, we note that for these values of $ k $, Lemma \ref{lemmaBD} gives us absolute upper bounds for $ n $. However, these upper bounds are so large that we wish to reduce them to a range where the solutions can be easily identified by a computer. To do this, we return to \eqref{fala3} and put
\begin{eqnarray}
\Gamma:=(n-1)\log\alpha-m\log 3+\log(f_k(\alpha)).
\end{eqnarray}
For technical reasons we assume that $ \min\{n-n_1, m-m_1\} \ge 20 $. In the case that this condition fails, we consider one of the following inequalities instead:
\begin{itemize}
\item [(i)] if $ n-n_1<20 $ but $ m-m_1\ge 20 $, we consider \eqref{fala6};
\item[(ii)] if $ n-n_1\ge 20 $ but $ m-m_1< 20 $, we consider \eqref{fala7};
\item[(iii)] if $ n-n_1<20 $ but $ m-m_1< 20 $, we consider \eqref{fala9}.
\end{itemize}
We start by considering \eqref{fala3}. Note that $ \Gamma\neq 0 $; thus we distinguish the following two cases. If $ \Gamma>0 $, then $ e^{\Gamma}-1>0 $, then from \eqref{fala3} we get
\begin{eqnarray*}
0<\Gamma<e^{\Gamma}-1< \max\left\{\alpha^{n_1-n+6}, 3^{m_1-m+1}\right\}.
\end{eqnarray*}
Next we suppose that $ \Gamma<0 $. Since $ \Lambda = |e^{\Gamma}-1|<\frac{1}{2} $, we get that $ e^{|\Gamma|}<2 $. Therefore,
\begin{eqnarray*}
0<|\Gamma|\le e^{|\Gamma|}-1 = e^{|\Gamma|}|e^{\Gamma}-1|<2\max\left\{\alpha^{n_1-n+6}, 3^{m_1-m+1}\right\}.
\end{eqnarray*}
Therefeore, in any case, the following inequality holds
\begin{eqnarray}\label{BD111}
0<|\Gamma|<2\max\left\{\alpha^{n_1-n+6}, 3^{m_1-m+1}\right\}.
\end{eqnarray}
By replacing $ \Gamma $ in the above inequality by its formula and dividing through by $ \log 3 $, we then conclude that
\begin{eqnarray*}
0<\left|(n-1)\left(\dfrac{\log\alpha}{\log 3}\right)-m+\dfrac{\log(f_k(\alpha))}{\log 3}\right|<\max\left\{(2\alpha^{6})\cdot\alpha^{-(n-n_1)}, \dfrac{6}{\log 3}\cdot 3^{-(m-m_1)}\right\}
\end{eqnarray*}
Then, we apply Lemma \ref{Dujjella} with the following data
\begin{eqnarray*}
k\in[4, 600], ~~~~ \tau_{k}:=\dfrac{\log\alpha}{\log 3}, ~~~~ \mu_{k}:=\dfrac{\log(f_k(\alpha))}{\log 3}, ~~~~(A_{k}, B_{k}):= (2\alpha^{6}, \alpha) \text{  or  } \left(\dfrac{6}{\log 3},3\right).
\end{eqnarray*}
Next, we put $ M_{k}:=\lfloor 4\times 10^{42}k^{11}(\log k)^{7} \rfloor $, which is the absolute upper bound on $ n $ by Lemma \ref{lemmaBD}.  An intensive computer search in \textit{Mathematica} revealed that the maximum value of $ \lfloor \log (2\alpha^6q/\varepsilon)/\log\alpha \rfloor $ is $ < 600$ and the maximum value of $ \lfloor \log ((6/\log3 )q/\varepsilon)/\log 3 \rfloor $ is $ <375 $. Thus, either
\begin{eqnarray*}
n-n_1 < \dfrac{\log(2\alpha^6q/\varepsilon)}{\log\alpha}< 600, ~~~\text{  or  } m-m_1 < \dfrac{\log((6/\log 3)q/\varepsilon)}{\log 3}< 375.
\end{eqnarray*}
Therefore, we have that either $ n-n_1 \leq 600 $ or $ m-m_1 \le  375 $.

Now, let us assume that $ n-n_1\leq 600 $. In this case, we consider the inequality \eqref{fala6} and assume that $ m-m_1 \geq 20 $. Then we put
\begin{eqnarray*}
\Gamma_{1}:=(n_1-1)\log\alpha -m\log 3+\log ((f_k(\alpha)(\alpha^{n-n_1}-1)).
\end{eqnarray*}
By similar arguments as in the previous step for proving \eqref{BD111}, from \eqref{fala6} we get
\begin{eqnarray*}
0<|\Gamma_1|<\dfrac{6}{3^{m-m_1}},
\end{eqnarray*}
and replacing $ \Gamma_1 $ with its formula and dividing through by $ \log 3 $ gives
\begin{eqnarray}\label{BD112}
0<\left|(n_1-1)\left(\dfrac{\log\alpha}{\log 3}\right)-m+\dfrac{\log(f_k(\alpha)(\alpha^{n-n_1}-1))}{\log 3}\right|<\dfrac{6}{\log 3}\cdot3^{-(m-m_1)}.
\end{eqnarray}
As before, we keep the same $ \tau_k, ~~M_k,~~ (A_k, B_k):=((6/\log 3),3) $ and put
\begin{eqnarray*}
\mu_{k,l}:=\dfrac{\log(f_k(\alpha)(\alpha^{l}-1))}{\log 3}, ~~~~k\in[4, 600], ~~~~l:=n-n_1 \in[1, 600].
\end{eqnarray*}
We apply Lemma \ref{Dujjella} to the inequality \eqref{BD112} with the above data. A computer search in \textit{Mathematica} revealed that the maximum value of $ \lfloor \log(Aq/\varepsilon)/\log B \rfloor $ over the values of $ k\in [4, 600] $ and $ l\in [1, 600] $ is $ < 377 $. Hence, $ m-m_1 \le 377 $.

Next, we assume that $ m-m_1 \le 375 $. Here, we consider the inequality \eqref{fala7} and also assume that $ n-n_1 \geq 20$. We put
\begin{eqnarray*}
\Gamma_{2}:=(n-1)\log\alpha-m_1\log 3+\log\left(f_k(\alpha)/(3^{m-m_1}-1)\right).
\end{eqnarray*}
Thus, by the same arguments as before, we get
\begin{eqnarray*}
0<|\Gamma_{2}|<\dfrac{2\alpha^{6}}{\alpha^{n-n_1}}.
\end{eqnarray*}
By substituting for $ \Gamma_{2} $ with its formula and dividing through by $ \log 3 $ in the above inequality, we get
\begin{eqnarray*}
0<\left|(n-1)\left(\dfrac{\log\alpha}{\log 3}\right)-m_1+\dfrac{\log\left(f_k(\alpha)/(3^{m-m_1}-1)\right)}{\log 3}\right|<\dfrac{2\alpha^6}{\log 3}\cdot\alpha^{-(n-n_1)}.
\end{eqnarray*}
As before, we apply Lemma \ref{Dujjella} with the same $ \tau_k,~~M_k, ~~(A_k, B_k):=(2\alpha^{6}/\log 3, \alpha) $ and put
\begin{eqnarray*}
\mu_{k,j}:=\dfrac{\log\left(f_k(\alpha)/(3^{m-m_1}-1)\right)}{\log 3}, ~~~~ k\in[4, 600], ~~~j:=m-m_1\in [1, 375].
\end{eqnarray*}
A computer search with \textit{Mathematica} revealed that the maximum value of $ \lfloor \log(Aq/\varepsilon)/\log B \rfloor$, for $ k\in[4, 600] $ and $ j\in [1, 375] $ is $ < 603 $. Hence, $ n-n_1 \leq 603 $.

To conclude the above computations, first we got that either $ n-n_1\leq 600 $ or $ m-m_1\leq 375 $. If $ n-n_1 \leq 600 $, then $ m-m_1 \leq  377 $, and if $ m-m_1 \leq 375 $, then $ n-n_1 \leq 603 $. Therefore, we can conclude that we always have
\begin{eqnarray*}
n-n_1\leq 603 ~~~\text{  and  } m-m_1 \leq 377.
\end{eqnarray*}

Finally, we go to \eqref{fala9} and put
\begin{eqnarray*}
\Gamma_{3}:=(n_1-1)\log\alpha - m_1\log 3 + \log\left(\dfrac{f_k(\alpha)(\alpha^{n-n_1}-1)}{3^{m-m_1}-1}\right).
\end{eqnarray*}
Since $ n>600 $, from \eqref{fala9} we can conclude that
\begin{eqnarray*}
0<|\Gamma_3|<\dfrac{2\cdot 3^{5}}{3^{0.8n}}.
\end{eqnarray*}
Hence, by substituting for $ \Gamma_{3} $ by its formula and dividing through by $ \log 3 $, we get
\begin{eqnarray*}
0<\left|(n_1-1)\left(\dfrac{\log\alpha}{\log 3}\right) - m_1 + \dfrac{\log\left(f_k(\alpha)(\alpha^{n-n_1}-1)/(3^{m-m_1}-1)\right)}{\log 3}\right|< 1328\cdot 3^{-0.8n}.
\end{eqnarray*}
We apply Lemma \ref{Dujjella} with the same $ \tau_k, ~~M_k, ~~(A_k, B_k):=(1328, 3), ~~k\in [4, 600], $
and put
\begin{eqnarray*}
\mu_{k, l, j}:=\dfrac{\log\left(f_k(\alpha)(\alpha^l-1)/(3^{j}-1)\right)}{\log 3},  ~~l:=n-n_1\in[1, 603], ~~j:=m-m_1\in [1, 377].
\end{eqnarray*}
A computer search in \textit{Mathematica} revealed that the maximum value of $ \lfloor \log(1328q/\varepsilon)/\log 3\rfloor $, for $ k\in[4,600], ~~l\in[1, 603]  $ and $ j\in [1, 377] $ is $ < 378 $. Hence, $ n< 473 $, which contradicts the assumption that $ n> 500 $ in the previous section.

\subsection{The case of large $k$}
We now assume that $ k>600 $. Note that for these values of $ k $ we have
\begin{eqnarray*}
n<4\times 10^{42}k^{11}(\log k)^{7}.
\end{eqnarray*}
Since, $ n\ge k+2 $, we have that $ n\ge 602 $.
The following lemma is useful.
\begin{lemma}\label{Kala224}
For $1\le n<2^{k/2}$ and $k\ge 10$, we have 
$$
F_n^{(k)}=2^{n-2}\left(1+\zeta\right)\quad {\text{where}}\quad |\zeta|<\frac{5}{2^{k/2}}.
$$
\end{lemma}

\begin{proof}
When $n\le k+1$, we have $F_{n}^{(k)}=2^{n-2}$ so we can take $\zeta:=0$. So, assume $k+2\le n<2^{k/2}$. It follows from (1.8) in \cite{BGL} that
$$
|f_k(\alpha)\alpha^{n-1}-2^{n-2}|<\frac{2^n}{2^{k/2}}.
$$
By \eqref{approxgap},  we also have $\left|F_{n}^{(k)}-f_k(\alpha)\alpha^{n-1}\right|<1/2$. Thus,
\begin{eqnarray*}
|F_n^{(k)}-2^{n-2}| & \le & |f_k(\alpha)\alpha^{n-1}-2^{n-2}|+|F_n^{(k)}-f_k(\alpha)\alpha^{n-1}|\\
& < & \frac{2^n}{2^{k/2}}+\frac{1}{2}=\frac{2^n}{2^{k/2}}\left(1+\frac{1}{2^{n-k/2+1}}\right)\le 
\frac{2^n}{2^{k/2}}\left(1+\frac{1}{2^{k/2+3}}\right)\\
& < & \frac{2^n\cdot 1.25}{2^{k/2}}=\left(\frac{5}{2^{k/2}}\right) 2^{n-2}.
\end{eqnarray*}
\end{proof}
By the above lemma, we can rewrite \eqref{fala1} as
$$
2^{n-2}(1+\zeta)-2^{n_1-2}(1+\zeta_1)=3^m-3^{m_1},\qquad \max\{|\zeta|,|\zeta_1|\}<\frac{5}{2^{k/2}}.
$$
So,
\begin{eqnarray}
\label{Kalai1}
|2^{n-2}-3^m| & = & |-\zeta\cdot 2^{n-2}+2^{n_1-2}(1+\zeta_1)-3^{m_1}|\nonumber \\
& \le & 2^{n-2}\left(\frac{5}{2^{k/2}}\right)+2^{n_1-2}\left(1+\frac{5}{2^{k/2}}\right)+3^{m_1}.
\end{eqnarray}
Next, we have 
$$
2^{n-2} > F_n^{(k)}-F_{n_1}^{(k)}=3^m-3^{m_1}\ge 2\cdot 3^{m-1},\quad {\text{\rm so}} \quad 2^{n-2}/3^m>2/3.
$$
Further,
\begin{eqnarray*}
3^m>3^m-3^{m_1} & = & F_n^{(k)}-F_{n_1}^{(k)} \ge  F_{n}^{(k)}-F_{n-1}^{(k)}\\
& \ge & F_{n-2}^{(k)}>2^{n-4}\left(1-\frac{5}{2^{k/2}}\right)\\
& > & 2^{n-4}\left(\frac{27}{32}\right)\quad (k>10),
\end{eqnarray*}
so 
\begin{equation}
\frac{128}{27}>\frac{2^{n-2}}{3^m}>\frac{2}{3}.
\end{equation}
Going back to \eqref{Kalai1}, we have
$$
|3^{m}2^{-(n-2)}-1|<\frac{5}{2^{k/2}}+\frac{1.25}{2^{n-n_1}}+\frac{3^{m_1}}{(2/3) 3^m}=\frac{5}{2^{k/2}}+1.5\left(\frac{1}{2^{n-n_1}}+\frac{1}{3^{m-m_1}}\right).
$$
Thus, 
\begin{eqnarray}\label{Kalai2}
|3^{m}2^{-(n-2)}-1|<8\max\left\{\frac{1}{2^{n-n_1}},\frac{1}{3^{m-m_1}}, \frac{1}{2^{k/2}}\right\}.
\end{eqnarray}
We now apply Theorem \ref{Matveev11} on the left-hand side of \eqref{Kalai2} with the data
\begin{eqnarray*}
\Gamma := 3^{m}2^{-(n-2)}-1,\quad t:=2, \quad \gamma_{1}:=3, \quad \gamma_{2}:=2, \quad b_{1}:=m, \quad b_{2}:=-(n-2).
\end{eqnarray*}
It is clear that $ \Gamma \neq 0 $, otherwise we would get $ 3^{m}=2^{n-2} $ which is a contradiction since $ 3^{m} $ is odd  while $ 2^{n-2} $ is even. 
We consider the field $ \mathbb{K}=\mathbb{Q} $, in this case $ D=1 $. Since $ h(\gamma_1)=h(3)=\log 3 $ and $ h(\gamma_2)=h(2)=\log 2 $, we can take $ A_{1}:=\log 3 $ and $ A_{2}:=\log 2 $. We also take $ B:=n $. Then, by Theorem \ref{Matveev11}, the left-hand side of \eqref{Kalai2} is bounded below as
\begin{eqnarray}
\log|\Gamma|>-5.86\times 10^{8}(1+\log n).
\end{eqnarray}
By comparing with \eqref{Kalai2}, we get
\begin{eqnarray*}
\min\{(n-n_1-3)\log 2, ~~(m-m_1-2)\log 3, ~~(k/2-3)\log 2\}<5.86\times 10^{8}(1+\log n),
\end{eqnarray*}
which implies that
\begin{eqnarray}
\min\{(n-n_1)\log2, ~~(m-m_1)\log 3, ~~(k/2)\log 2\}<5.88\times 10^{8}(1+\log n).
\end{eqnarray}
Now the argument is split into four cases.

\medskip

\textbf{Case 5.3.1.} $ \min\{(n-n_1)\log2, ~~(m-m_1)\log 3, ~~(k/2)\log 2\}=(k/2)\log2 $.

\medskip 

\noindent In this case, we have
\begin{eqnarray*}
(k/2)\log 2<5.88\times 10^{8}(1+\log n),
\end{eqnarray*}
which implies that
\begin{eqnarray*}
k<1.70\times 10^{9}(1+\log n).
\end{eqnarray*}

\medskip

\textbf{Case 5.3.2.} $ \min\{(n-n_1)\log2, ~~(m-m_1)\log 3, ~~(k/2)\log 2\}=(n-n_1)\log2 $.

\medskip 
\noindent We rewrite \eqref{fala1} as
\begin{eqnarray*}
|3^{m}-2^{n_1-2}(2^{n-n_1}-1)|&=&|3^{m_1}+2^{n-2}\zeta-2^{n_1-2}\zeta_1|\\
&<&3^{m_1}+2^{n-2}\left(\dfrac{10}{2^{k/2}}\right),
\end{eqnarray*}
which implies that
\begin{eqnarray}\label{Kalai3}
\left|3^{m}2^{-n_1}(2^{n-n_1}-1)^{-1}-1\right|<20\max\left\{\dfrac{1}{3^{m-m_1}}, \dfrac{1}{2^{k/2}}\right\}.
\end{eqnarray}
We now apply Matveev's theorem,  Theorem \ref{Matveev11} on the left-hand side of \eqref{Kalai3} to 
$$ 
\Gamma_{1}=3^{m}2^{-(n_1-2)}(2^{n-n_1}-1)^{-1}-1,
$$
\begin{eqnarray*}
 t:=3, \quad \gamma_{1}:=3, \quad \gamma_{2}:=2, \quad \gamma_{3}:=2^{n-n_1}-1,\quad b_1:=m, \quad b_2:=-(n_1-2), \quad b_3:=-1.
\end{eqnarray*}
Note that $ \Gamma_{1} \neq 0 $. Otherwise, $ 3^{m}=2^{n-2}-2^{n_1-2} $, so $n_1=2$, and $2^{n-2}-3^m=1$, so $n\le 4$ by classical results on Catalan's equation,  which is a contradiction  
because $ n\geq k+2>602 $.  We use the same values, $ A_1:=\log 3 $, $ A_2:=\log 2 $, $ B:=n $ as in the previous step. In order to find $A_3$, note that
\begin{eqnarray*}
h(\gamma_3)=h(2^{n-n_1}-1)\leq (n-n_1+1)\log 2 < 5.90\times 10^{8}(1+\log n).
\end{eqnarray*}
So, we take $ A_3:=5.90\times 10^{8}(1+\log n) $. By Theorem \ref{Matveev11}, we have
\begin{eqnarray*}
\log|\Gamma_{1}|>-6.43\times 10^{19}(1+\log n)^{2}.
\end{eqnarray*}
By comparing with \eqref{Kalai3}, we get
\begin{eqnarray*}
\min\{(m-m_1-3)\log3, ~(k/2-5)\log 2\}<6.43\times 10^{19}(1+\log n)^{2},
\end{eqnarray*}
which implies that
\begin{eqnarray*}
\min\{(m-m_1)\log 3, (k/2)\log 2\}<6.44\times 10^{19}(1+\log n)^{2}.
\end{eqnarray*}
At this step, we have that either $$ (m-m_1)\log 3 < 6.44\times 10^{19}(1+\log n)^{2} $$ or $$ k<1.86\times 10^{20}(1+\log n)^{2}. $$

\medskip

\textbf{Case 5.3.3.} $ \min\{(n-n_1)\log2, ~~(m-m_1)\log 3, ~~(k/2)\log 2\}=(m-m_1)\log3 $.

\medskip

\noindent We rewrite \eqref{fala1} as
\begin{eqnarray*}
|(3^{m_1}(3^{m-m_{1}}-1)-2^{n-2}|&=&|2^{n-2}\zeta-2^{n_1-2}(1+\zeta_1)|\\
&<&2^{n-2}\left(\frac{5}{2^{k/2}}\right)+2^{n_1-2}\left(1+\frac{5}{2^{k/2}}\right),
\end{eqnarray*}
which implies that
\begin{eqnarray}\label{Kalai4}
\left|3^{m_1}(3^{m-m_1}-1)2^{-(n-2)}-1\right|<20\max\left\{\dfrac{1}{2^{n-n_1}}, \dfrac{1}{2^{k/2}}\right\}.
\end{eqnarray}
We again apply Matveev's theorem,  Theorem \ref{Matveev11} on the left-hand side of \eqref{Kalai3} which is
$$ 
\Gamma_{2}=3^{m_1}2^{-(n-2)}(3^{m-m_1}-1)-1,
$$
\begin{eqnarray*}
 t:=3, \quad \gamma_{1}:=3, \quad \gamma_{2}:=2, \quad \gamma_{3}:=(3^{m-m_1}-1), \quad b_1:=m_1, \quad b_2:=-(n-2), \quad b_3:=1.
\end{eqnarray*}
Note that $ \Gamma_2 \neq 0 $. Otherwise, $ 3^{m}-3^{m_1}=2^{n-2}$, which is impossible since the left--hand side is a multiple of $3$ and the right--hand side isn't.
 We use the same values, $ A_1:=\log 3 $, $ A_2:=\log 2 $, $ B:=n $ as in the previous steps. In order to determine $A_3$, note that
\begin{eqnarray*}
h(\gamma_3)=h(3^{m-m_1}-1)\leq (m-m_1+1)\log 3 < 5.90\times 10^{8}(1+\log n).
\end{eqnarray*}
So, we take $ A_3:=5.90\times 10^{8}(1+\log n) $. By Theorem \ref{Matveev11}, we have the lower bound
\begin{eqnarray*}
\log|\Gamma_{2}|>-6.43\times 10^{19}(1+\log n)^{2}.
\end{eqnarray*}
By comparing with \eqref{Kalai4}, we get
\begin{eqnarray*}
\min\{(n-n_1-5)\log3, ~(k/2-5)\log 2\}<6.43\times 10^{19}(1+\log n)^{2},
\end{eqnarray*}
which implies that
\begin{eqnarray*}
\min\{(n-n_1)\log 3, (k/2)\log 2\}<6.44\times 10^{19}(1+\log n)^{2}.
\end{eqnarray*}
As before, at this step we have that either $$ (n-n_1)\log 3 < 6.44\times 10^{19}(1+\log n)^{2} $$ or $$ k< 1.86\times 10^{20}(1+\log n)^{2}.  $$

\medskip

Therefore, in all the three cases above, we got
\begin{eqnarray}
\min\{(n-n_1)\log2, ~~(m-m_1)\log 3, ~~(k/2)\log 2\}&<&5.88\times 10^{8}(1+\log n)\nonumber\\
\max\{(n-n_1)\log2, ~~(m-m_1)\log 3, ~~(k/2)\log 2\} &<&6.44\times 10^{19}(1+\log n)^{2}.
\end{eqnarray}

\medskip

\textbf{Case 5.3.4.} $(k/2)\log 2>6.44\times 10^{19}(1+\log n)^{2}$.

\medskip

From the previous analysis, we conclude that one of $(n-n_1)\log 2$ and $(m-m_1)\log 3$ is bounded by $5.88\times 10^{8}(1+\log n)$ and the other one by 
$6.44\times 10^{19}(1+\log n)^{2}.$ We rewrite \eqref{fala1} as
\begin{eqnarray*}
\left|3^{m_1}(3^{m-m_1}-1)-2^{n_1-2}(2^{n-n_1}-1)\right|=|\zeta|\cdot 2^{n-2}+|\zeta_1|\cdot 2^{n_1-2}\le 2^{n-2} \left(\frac{10}{2^{k/2}}\right),
\end{eqnarray*}
which implies that
\begin{eqnarray}\label{Kalai5}
\left|3^{m_{1}}2^{-(n_1-2)}\left(\dfrac{3^{m-m_1}-1}{2^{n-n_1}-1}\right)-1\right|<\dfrac{20}{2^{k/2}}.
\end{eqnarray}
We apply Matveev's Theorem to 
$$ 
\Gamma_{3}=3^{m_{1}}2^{-(n_1-2)}\left(\dfrac{3^{m-m_1}-1}{2^{n-n_1}-1}\right)-1,
$$
with the data 
\begin{eqnarray*}
t=:3, \quad \gamma_{1}:=3, \quad \gamma_{2}:=2, \quad \gamma_{3}:=\left(\dfrac{3^{m-m_1}-1}{2^{n-n_1}-1}\right), \quad b_1:=m_1, \quad b_2:=-(n_1-2),\quad  b_3:=1.
\end{eqnarray*}
Note that $ \Gamma_{3}\neq 0 $, otherwise, we get $ 2^{n}-3^{m}=2^{n_1}-3^{m_1}$ which is impossible by Lemma \ref{classical1}. 

 As before we take $ B:=n $, $ A_1:=\log 3 $, $ A_2:=\log 2 $. In oder to determine an acceptable value for $ A_3 $, note that
\begin{eqnarray*}
h(\gamma_{3})&\leq& h(3^{m-m_1}-1)+h(2^{n-n_1}-1)<(m-m_1+1)\log 3+(n-n_1+1)\log 2\\
&<&2\times 6.46\times 10^{19}(1+\log n)^{2}<1.30\times 10^{20}(1+\log n)^{2}.
\end{eqnarray*}
Thus, we take $ A_3:=1.30\times 10^{20}(1+\log n)^{2} $. By Theorem \ref{Matveev11}, we have
\begin{eqnarray*}
\log|\Gamma_{3}|>-1.86\times 10^{31}(1+\log n)^{3}.
\end{eqnarray*}
By comparing with \eqref{Kalai5}, we get
\begin{eqnarray*}
(k/2-5)\log 2 < 1.86\times 10^{31}(1+\log n)^{3},
\end{eqnarray*}
which implies that
\begin{equation}
\label{eq:finalk}
k<5.42\times 10^{31}(1+\log n)^{3}.
\end{equation}
Thus, inequality \eqref{eq:finalk} holds in all four cases. Since $ n<4\times 10^{42}k^{11}(\log k)^{7} $, then 
\begin{eqnarray}
k<5.42\times 10^{31}\left(1+\log\left(4\times 10^{42}k^{11}(\log k)^{7}\right)\right)^{3},
\end{eqnarray}
which gives the absolute upper bounds $$ k<8.631\times 10^{40}~<10^{41} $$ and $$ m<n<3.44\times 10^{506}<10^{507}. $$

We record what we have proved.

\begin{lemma}
\label{lem:abs}
We have
$$
k<10^{41}\qquad {\text{and}}\qquad m<10^{507}.
$$
\end{lemma}

\subsection{The final reduction}

The previous bounds are too large, so we need to reduce them by applying a Baker-Davenport reduction procedure. First, we go to \eqref{Kalai2} and let
\begin{eqnarray*}
z:=m\log 3 - (n-2)\log 2.
\end{eqnarray*}
Assume $m-m_1>1066$, $ n-n_1>1690 $  and $ k>600 $. 
Then, we note that \eqref{Kalai3} can be rewritten as
\begin{eqnarray*}
\left|e^{z}-1\right|<\max\{2^{n_1-n+3}, ~~3^{m_1-m+2}, ~~2^{-k/2+3}\}.
\end{eqnarray*}
If $ z>0 $, then $ e^{z}-1 >0$, so we obtain
\begin{eqnarray*}
0<z<e^{z}-1<\max\{2^{n_1-n+3}, ~~3^{m_1-m+2}, ~~2^{-k/2+3}\}.
\end{eqnarray*}
Suppose now that $ z<0 $. Since $ \Gamma=|e^{z}-1|<1/2 $, we get that $ e^{|z|}<2 $. Thus,
\begin{eqnarray*}
0<|z|\leq e^{|z|}-1=e^{|z|}|e^{z}-1|<2\max\{2^{n_1-n+3}, ~~3^{m_1-m+2}, ~~2^{-k/2+3}\}.
\end{eqnarray*}
Therefore, in any case we have that the inequality
\begin{eqnarray}
0<|z|<2\max\{2^{n_1-n+3}, ~~3^{m_1-m+2}, ~~2^{-k/2+3}\}
\end{eqnarray}
always holds. By replacing $ z $ in the above inequality by its formula and dividing through by $ m\log 2 $, we get that
\begin{eqnarray}
0<\left|\dfrac{\log 3}{\log 2}-\dfrac{n}{m}\right|<\max\left\{\dfrac{24}{2^{n-n_1}m},~~\dfrac{26}{3^{m-m_1}m}, ~~\dfrac{24}{2^{k/2}m}\right\}.
\end{eqnarray}
Then
\begin{eqnarray*}
\max\left\{\dfrac{24}{2^{n-n_1}m},~~\dfrac{26}{3^{m-m_1}m}, ~~\dfrac{24}{2^{k/2}m}\right\}<\dfrac{1}{2m^{2}},
\end{eqnarray*}
because $ m<10^{507} $.
By the Legendre criterion Lemma \ref{lem:legendre}, it follows that $ n/m $ is a convergent of $ \log3/\log2 $. So $ n/m $ is of the form $ p_l/q_l $ for some $ l=0, 1, 2, \ldots, 972 $. Then $n/m=p_l/q_l$ implies that $m=dq_l$ for some $d\ge 1$. Thus,
\begin{eqnarray*}
\dfrac{1}{(a_{l+1}+2)q_{k}q_{l+1}}<\left|\dfrac{\log 3}{\log 2}-\dfrac{p_l}{q_l}\right|<\max\left\{\dfrac{24}{2^{n-n_1}dq_l},~~\dfrac{26}{3^{m-m_1}dq_l}, ~~\dfrac{24}{2^{k/2}dq_l}\right\}.
\end{eqnarray*}
Since $ \max\{a_{l+1}: l=0, 1, 2, \ldots, 972]=3308 $, we get that
\begin{eqnarray*}
\min\{2^{n-n_1}, 3^{m-m_1}, 2^{k/2}\}\le 26\cdot3310q_{973}.
\end{eqnarray*}
With the help of \textit{Mathematica}, we have $ q_{973} \approx 1.6834\times 10^{507} $. We then conclude that one of the following inequalities holds:
\begin{eqnarray*}
n-n_1<1690, \qquad m-m_1<1066, \qquad k<3380.
\end{eqnarray*}

Suppose first that $ m-m_1>10 $ and $ k\ge 20 $, we go back to \eqref{Kalai3} and let
\begin{eqnarray}
z_1:=m\log 3-(n_1-2)\log 2-\log(2^{n-n_1}-1).
\end{eqnarray}
Then we note that \eqref{Kalai3} can be rewritten as
\begin{eqnarray*}
\left|e^{z_1}-1\right|<\max\{3^{m_1-m+3}, ~~2^{-k/2+5}\}.
\end{eqnarray*}
This implies that
\begin{eqnarray*}
0<|z_1|<2\max\{3^{m_1-m+3}, ~~2^{-k/2+5}\}.
\end{eqnarray*}
This also holds when $ m-m_1<10 $ and $k<20$. By substituting for $ z_1 $ and dividing through by $ \log 2 $, we get
\begin{eqnarray*}
0<\left|m\left(\dfrac{\log 3}{\log 2}\right)-(n_1-2)+\dfrac{\log(1/(2^{n-n_1}-1))}{\log 2}\right|<\max\{98\cdot3^{-(m-m_1)}, ~~94\cdot2^{-k/2}\}.
\end{eqnarray*}
We put 
\begin{eqnarray*}
\tau: = \dfrac{\log 3}{\log 2}, \qquad \mu: = \dfrac{\log(1/(2^{n-n_1}-1))}{\log 2},\qquad  (A,B):=(78,3) \quad \text{ or } \quad (94,2),
\end{eqnarray*}
where $n-n_1 \in [1, 1690]$. We take $M:=10^{507}$. A computer search in \textit{Mathematica} reveals that $ q=q_{977}\approx 5.708\times 10^{510}> 6M $ and the minimum positive value of $ \varepsilon:=||\mu q||-M||\tau q||>0.0186 $. Thus, Lemma \ref{Dujjella} tells us that either $ m-m_1\leq 1078 $ or $ k\leq 3418 $.

Next, we suppose that $ n-n_1 > 10 $, $ k>20 $ and  go to \eqref{Kalai4} and let
\begin{eqnarray}
z_2:=m_1\log 3-(n-2)\log 2+\log(3^{m-m_1}-1).
\end{eqnarray}
Then we also  note that \eqref{Kalai4} can be rewritten as
\begin{eqnarray*}
\left|e^{z_2}-1\right|<\max\{2^{n_1-n+5}, ~~2^{-k/2+5}\}.
\end{eqnarray*}
This gives
\begin{eqnarray*}
0<|z_2|<2\max\{2^{n_1-n+5}, ~~2^{-k/2+5}\}.
\end{eqnarray*}
This also holds for $ n-n_1< 10 $ and $ k<20 $ as well.
By substituting for $ z_2 $ and dividing through by $ \log 2 $, we get
\begin{eqnarray*}
0<\left|m_1\left(\dfrac{\log 3}{\log 2}\right)-(n-2)+\dfrac{\log(3^{m-m_1}-1)}{\log 2}\right|<\max\{94\cdot2^{-(n-n_1)}, ~~94\cdot2^{-k/2}\}.
\end{eqnarray*}
We put 
\begin{eqnarray*}
\tau: = \dfrac{\log 3}{\log 2}, \qquad\mu: = \dfrac{\log(3^{m-m_1}-1)}{\log 2},\qquad (A,B):=(94,2),
\end{eqnarray*}
where $m-m_1 \in [1, 1066]$. We keep the same $ M $ and $ q $ as in the previous step. A computer search in \textit{Mathematica} reveals that the minimum positive value of $ \varepsilon:=||\mu q||-M||\tau q||>0.0372 $. Thus, Lemma \ref{Dujjella} tells us that either $ n-n_1\leq 1708 $ or $ k\leq 3416 $.

Lastly, we assume that $ k>20 $ and go to \eqref{Kalai5} and let
\begin{eqnarray}
z_3:=m_1\log 3-(n_1-2)\log 2-\log((3^{m-m_1}-1)/(2^{n-n_1}-1)).
\end{eqnarray}
We note that \eqref{Kalai5} can be rewritten as
\begin{eqnarray*}
\left|e^{z_3}-1\right|<2^{-k/2+5}.
\end{eqnarray*}
This gives
\begin{eqnarray*}
0<|z_3|<2^{-k/2+6},
\end{eqnarray*}
which also holds when $ k<20 $. By substituting for $ z_3 $ and dividing through by $ \log 2 $, we get
\begin{eqnarray*}
0<\left|m_1\left(\dfrac{\log 3}{\log 2}\right)-(n_1-2)+\dfrac{\log((3^{m-m_1}-1)/(2^{n-n_1}-1))}{\log 2}\right|<94\cdot2^{-k/2}.
\end{eqnarray*}
We put 
\begin{eqnarray*}
\tau: = \dfrac{\log 3}{\log 2}, \qquad \mu: = \dfrac{\log((3^{m-m_1}-1)/(2^{n-n_1}-1))}{\log 2},\qquad (A,B):=(94,2),
\end{eqnarray*}
where $n-n_1 \in [1, 1708]$ and $ m-m_1 \in [1, 1074] $. We keep the same $ M $ and  $ q $ as before. A computer search in \textit{Mathematica} reveals that the minimum positive value of $ \varepsilon:=||\mu q||-M||\tau q||>0.00058 $. Thus, Lemma \ref{Dujjella} tells us that $ k\leq 3428 $. 

Therefore, in all cases we found out that $ k<3428 $ which gives that $ n<7.2741\times10^{87}<10^{88} $. These bounds are still too large. We repeat the above procedure several times by adjusting the values of $ M $ with respect to the new bounds of $ n $. We summarise the data for the iterations performed in Table \ref{tab1}
\begin{table}[H]
\caption{Computation results}\label{tab1}
\begin{center}
\begin{tabular}{|c|c|c|c|c|}
\hline
 & $M$& $n-n_1 \le$ & $m-m_1 \le$ & $k\le$\\
\hline
$ 1 $& $ ~10^{507} $& $ 1708 $& $ 1074 $& $ 3428 $\\
$ 2 $& $ 10^{88} $& $ ~319 $& $ ~197 $& $ ~662 $\\
$ 3 $& $ 10^{80} $& $ ~287 $& $ ~180 $& $ ~590 $\\
$ 4 $& $ 10^{79} $& $ ~282 $& $ ~180 $& $ ~584 $\\
$ 5 $& $ 10^{79} $& $ ~282 $& $ ~180 $& $ ~584 $\\
\hline
\end{tabular}
\end{center}
\end{table}
\noindent
From the data displayed in the above table, it is evident that  after four times of the iteration, the upper bound on $ k $ stabilizes at $ 584 $.  Hence, $ k<600 $ which contradicts our assumption that $ k>600 $. Therefore, we have no further solutions to the Diophantine equation \eqref{Problem} with $ k>600 $.

\section*{Acknowledgements}
M.~D. was supported by the Austrian Science Fund (FWF) grants: F5510-N26---Part of the special research program (SFB), ``Quasi-Monte Carlo Methods: Theory and Applications'' and W1230---``Doctoral Program Discrete Mathematics''. F.~L. was also supported by grant CPRR160325161141 from the NRF of South Africa, grant RTNUM19 from CoEMaSS, Wits, South Africa. Part of the work in this paper was done when both authors visited the Max Planck Institute for Mathematics Bonn, in March 2018 and the Institut de Math\'ematiques de Bordeaux, Universit\'e de Bordeaux, in May 2019. They thank these institutions for hospitality and fruitful working environments.

 \end{document}